\documentclass{article}
\usepackage{amsmath,amsfonts,amssymb,amsthm,mathtools}
\usepackage{fullpage}
\usepackage{hyperref}
\usepackage{enumerate}
\usepackage[numbers]{natbib}

\usepackage{fixmath}
\usepackage{graphicx}
\usepackage{subfig}
\usepackage{float}

\usepackage{tikz}
\usetikzlibrary{graphs}
\usetikzlibrary{angles,quotes}
\usetikzlibrary{calc}
\usetikzlibrary{arrows}
\usetikzlibrary{arrows.meta}
\usetikzlibrary{decorations.markings}
\usetikzlibrary{fit}
\usetikzlibrary{positioning}
\usepackage{tikz-cd}

\newcommand{\GL}{\mathrm{GL}}

\newcommand{\PGL}{\mathrm{PGL}}

\newcommand{\T}{\mathrm{T}}

\newcommand{\F}{\mathbb{F}}

\newcommand\yang[1]{}

\usepackage{amsbsy}
\newcommand\bs[1]{\boldsymbol{ #1}}

\newtheorem{definition}{Definition}[section]

\newtheorem{defn}[definition]{Definition}
\newtheorem{thm}[definition]{Theorem}
\newtheorem{prop}[definition]{Proposition}
\newtheorem{lem}[definition]{Lemma}
\newtheorem{cor}[definition]{Corollary}

\newtheorem{conj}[definition]{Conjectures}
\theoremstyle{definition}

\title{Reduced Power Graphs of $PGL_3(\F_q)$}
\author{Yilong Yang}
\date{\today}

\begin{document}
\maketitle
%\tableofcontents

%%%%%%%%%%%%%%%%%%%%
%%%%%%%%%%%%%%%%%%%%
%%%%%%%%%%%%%%%%%%%%

\abstract{
Given a group $G$, let us connect two non-identity elements by an edge if and only if one is a power of another. This gives a graph structure on $G$ minus identity, called the reduced power graph.

In this paper, we shall find the exact number of connected components and the exact diameter of each component for the reduced power graphs of $PGL_3(\F_q)$ for all prime power $q$.
}

\section{Introduction}

%\subsection{Background and Main Results}

Given a group $G$, the \emph{reduced power graph} of $G$ has all non-identity elements of $G$ as vertices, and two distinct vertices are connected by an edge if and only if one is a power of the other. One interesting aspect of the reduced power graph is its connectivity. In particular, Akbari and Ashrafi proposed the following conjecture in 2015.

\begin{conj}[Akbari and Ashrafi,\cite{AA15}]
The reduced power graph of a non-abelian simple group $G$ is connected only if $G$ is isomorphic to some alternating group $A_n$.
\end{conj}

In a previous work \cite{previous}\yang{cite!!!!}, we showed that the conjecture above is false, and have provided a description about the connected components of the reduced power graphs of $\PGL_n(\F_q)$ and an upper bound for these components. In this paper, we shall find the precise diameter bounds when $n=3$. In particular, we have the following result:

\begin{thm}
\label{MainThm:3}
Let $p$ be the characteristic of a finite field $\F_q$. The reduced power graph of $\PGL_3(\F_q)$ must fall into one of the following cases.
\begin{enumerate}
\item The reduced power graph of $\PGL_3(\F_2)$ has $51$ components, each with diameter $1$. \yang{type 788, type 4, type 5}
\item The reduced power graph of $\PGL_3(\F_3)$ has $321$ components. One has diameter $11$, and $312$ components have diameter $1$. and $8$ components have diameter $2$.  \yang{main, type 8, type 5}
\item If $q\neq 2$ is even and $q-1$ is prime, then the reduced power graph of $\PGL_3(\F_q)$ has $1+q^3(q^3-1)(q+1)(q-3)+\frac{1}{3}q^3(q^2-1)(q-1)$ components. One has diameter $13$, $q^3(q^3-1)(q+1)(q-3)$ components has diameter $1$, and the other $\frac{1}{3}q^3(q^2-1)(q-1)$ components have diameter $1$ if $q=8$, and diameter $2$ otherwise. \yang{main, type 1, type 5 depending on $q=8$ or not.}
\item If $q\neq 2$ is even and $q-1$ is not prime, then the reduced power graph of $\PGL_3(\F_q)$ has $1+\frac{1}{3}q^3(q^2-1)(q-1)$ components. One has diameter $10$, and the other $\frac{1}{3}q^3(q^2-1)(q-1)$ components have diameter $2$. \yang{main, type 5}
\item If $q\neq 3$ is odd and $q-1$ is a prime power (i.e., when $q=9$ or $q$ is a Fermat prime), then the reduced power graph of $\PGL_3(\F_q)$ has $1+\frac{1}{p-1}(q^3-1)(q^3-q)+\frac{1}{3}q^3(q^2-1)(q-1)$ components. One has diameter $12$, $\frac{1}{p-1}(q^3-1)(q^3-q)$ components has diameter $1$, and the other $\frac{1}{3}q^3(q^2-1)(q-1)$ components have diameter $1$ if $q^2+q+1$ is a prime, and diameter $2$ otherwise. \yang{main, type 8, type 5 depending on $q^2+q+1$ is prime or not}
\item If $q\neq 3$ is odd and $q-1$ is not a prime power, then the reduced power graph of $\PGL_3(\F_q)$ has $1+\frac{1}{p-1}(q^3-1)(q^3-q)+\frac{1}{3}q^3(q^2-1)(q-1)$ components. One has diameter $8$, all other components have diameter $2$. \yang{main, type 8, type 5}
\end{enumerate}
\end{thm}

\section{Preliminary}

\subsection{Projectively Reduced Power Graph and Pivot Component}

In this paper, we shall study a special graph on $\GL_n(q)$, which has the reduced power graph of $\PGL_n(q)$ as a quotient graph.

\begin{defn}
Given a group $G$, let $Z$ be its center. Then the projectively reduced power graph of $G$ is obtained from the reduced power graph of $G$ by deleting vertices in $Z$. 
\end{defn}

\begin{prop}
Given a group $G$, let $Z$ be its center. Let $\Gamma$ be projective reduced power graph of $G$, and let $\Gamma'$ be the reduced power graph of $G/Z$. If $\Gamma$ is connected, then $\Gamma'$ is connected with the same or less diameter.
\end{prop}
\begin{proof}
$\Gamma'$ is a quotient graph of $\Gamma$.
\end{proof}

When $q>2$ is a prime power, we say a matrix $A\in\GL_3(\F_q)$ is a \emph{pivot matrix} if it is diagonalizable over $\F_q$ with exactly two distinct eigenvalues. In the projectively reduced power graph of $G$, the connected component containing a pivot matrix is called a \emph{pivot component}. We also call the corresponding connected component in $\PGL_3(\F_q)$ a \emph{pivot component} as well. In the previous work \cite{previous}, we have shown that for $q\neq 2$, all pivot matrices are connected in the projectively reduced power graph of $\GL_3(\F_q)$. So in particular, when $q\neq 2$, the pivot component of $\GL_3(\F_q)$ or $\PGL_3(\F_q)$ is unique. Furthermore, the precise diameter of all non-pivot components of $\PGL_3(\F_q)$ is already obtained in \cite{previous}. So we only need to figure out the precise diameter of the pivot component.

\subsection{Obstructions and Connections}

We now list some important results from \cite{previous} about obstructions to the connectivity of the reduced power graph, and about certain connection results. We say a matrix $A\in\GL_n(\F_q)$ is a \emph{Jordan pivot matrix} if $A-I$ has rank one and $(A-I)^2=0$.

\begin{prop}[Diagonalizable obstruction]
\label{prop:ConnObstDiag}
Suppose $q>2$ is a power of $2$ with $q-1$ prime, and $2\leq n<q$. Let $A$ be an $n\times n$ diagonalizable matrix with $n$ distinct non-zero eigenvalues over $\F_q$. Then the image of $A$ in the reduced power graph of $\PGL_n(q)$ is trapped in a connected component with $q-2$ vertices and diameter $1$, made of images of non-identity powers of $A$.
\end{prop}

\begin{prop}[Irredicible obstruction]
\label{prop:ConnObstIrred}
Let $n$ be prime and $A$ be a matrix over $\F_q$ whose characteristic polynomial is irreducible. Then the image of $A$ in the reduced power graph of $\PGL_n(q)$ is trapped in a connected component with $(q^n-q)/(q-1)$ vertices, made of images of polynomials of $A$ that are not multiples of identity. The diameter of this component is at most $2$.
\end{prop}

\begin{prop}[Jordan-type obstruction]
\label{prop:ConnObstJordan}
Consider a finite field $\F_q$ with characteristic $p$. If $2\leq n\leq p$, let $A$ be any matrix over $\F_q$ similar to the $n\times n$ matrix $\begin{bmatrix}\lambda&1&&\\ &\ddots&\ddots&\\&&\ddots&1\\ &&&\lambda\end{bmatrix}$ for some $\lambda\in\F_q^*$. Then the image of $A$ in the reduced power graph of $\PGL_n(q)$ is trapped in a connected component with $p-1$ vertices and diameter $1$, made of images of powers of $A$ that are not multiples of identity.
\end{prop}

\begin{prop}[Quasi-diagonalizable obstruction when $q=2$]
\label{prop:ConnObstDiag2}
Suppose we have a prime $p_0=2^{p_1}-1$ (i.e., Mersenne prime) for some positive integer $p_1$. 
Suppose $n\leq p_0$. 
Let $A$ be an $n\times n$ non-identity matrix such that $A^{p_0}=I$, and its minimal polynomial equal to its characteristic polynomial. Then the $A$ in the reduced power graph of $\GL_n(2)$ is trapped in a connected component with $p_0-1$ vertices and diameter $1$, made of non-identity powers of $A$.

Note that in this case, $p_1$ must be a prime factor of $n$ or $n-1$.
\end{prop}

\begin{prop}[Extra irredicible obstruction when $q=2$]
\label{prop:ConnObstIrred2}
Let $n-1$ be prime and $A$ be an $n\times n$ matrix over $\F_2$ whose characteristic polynomial has an irreducible factor of degree $n-1$. Then the image of $A$ in the reduced power graph of $\PGL_n(2)$ is trapped in a connected component with $2^{n-1}-2$ vertices, made of non-identity powers of a matrix $C\in\PGL_n(2)$. The diameter of this component is at most $2$.
\end{prop}

\begin{prop}
\label{Prop:Pivot3}
For $q\neq 2$ and $n=3$, any two pivot matrices in the projectively reduced power graph of $\GL_n(\F_q)$ will have distance at most $8$.
\end{prop}

\begin{prop}
\label{Prop:PivotJordan}
For $q\neq 2$ and $n\geq 3$ or for $q=2$ and $n\geq 4$, any two Jordan pivot matrices in the projectively reduced power graph of $\GL_n(\F_q)$ will have distance at most $8$.
\end{prop}

\subsection{Other useful results}

We now list some important results that will be useful here. Proofs and references will be found in \cite{previous}.

\begin{lem}
\label{Lem:ConseqPP}
If $q,q-1$ are both prime powers, then either $q$ is Fermat prime, or $q-1$ is a Mersenne prime, or $q=9$.
\end{lem}

\begin{cor}
\label{Cor:IrredJordanCoprimeRoot}
Let $q$ be a power of a prime $p$, and let $n\geq 1$. Let $p_0$ be a prime factor of $q-1$. If any $n\times n$ matrix $A$ over $\F_q$ has multiplicative order coprime to $p_0$, then $A$ has a $p_0$-th root whose multiplicative order is a multiple of $p_0$.
\end{cor}

\begin{lem}
\label{Lem:NonJordayTypeClassify}
Let $q$ be a power of a prime $p$, and let $n\geq 2$. If a matrix $A\in\GL_n(\F_q)-Z(\GL_n(\F_q))$ has projective order coprime to $p$, then its generalized Jordan canonical form is a block diagonal matrix where each diagonal block is either $1\times 1$ or some companion matrix of an irreducible polynomial.
\end{lem}

\begin{lem}
\label{Lem:CompanionPower}
Suppose the characteristic polynomial of a matrix $C\in\GL_n(\F_q)$ is irreducible, and the multiplicative order of $C$ is $k$. For any positive integer $t$, the minimal polynomial of $C^t$ is irreducible. For any factor $t$ of $k$, if $\frac{k}{t}$ is a factor of $q^m-1$, then the minimal polynomial of $C^t$ is irreducible with degree at most $m$.
\end{lem}

\begin{cor}
\label{Cor:CompanionPower}
Suppose the characteristic polynomial of a matrix $C\in\GL_n(\F_q)$ is irreducible, and the multiplicative order of $C$ is $k$. For any factor $t$ of $k$, $\frac{k}{t}$ is a factor of $q-1$ if and only if $C^t$ is a multiple of identity.
\end{cor}

\section{Classification and Connections}

By standard enumeration, here is a list of possible types of generalized Jordan canonical forms for matrices in $\GL_3(\F_q)$.

\begin{enumerate}
\item (Multiples of identity) $A$ is a multiplies of identity. These are excluded in the reduced power graph of $\PGL_3(\F_q)$.
\item (Pivot matrices) The canonical form of $A$ is $\begin{bmatrix}a&&\\&a&\\&&b\end{bmatrix}$ diagonalizable and $a\neq b$. (Only possible if $q\neq 2$.)
\item (LLL matrices) The canonical form of $A$ is $\begin{bmatrix}a&&\\&b&\\&&c\end{bmatrix}$ all eigenvalues are distinct. (Only possible if $q\neq 2,3$.) We call them LLL matrices because the most important invaraint subspaces for this kind of matrices are the three eigenlines. By Proposition~\ref{prop:ConnObstDiag}, when $q-1$ is an odd prime, images of these matrices form isolated connected components in the reduced power graphs of $\PGL_3(\F_q)$, and each component has diameter at most $2$.
\item (Irreducible Matrices) The canonical form of $A$ is the companion matrix to an irreducible polynomial of degree three over $\F_q$. By Proposition~\ref{prop:ConnObstIrred}, images of these matrices form isolated connected components in the reduced power graphs of $\PGL_3(\F_q)$, and each component has diameter at most $1$.
\item (LP matrices) The canonical form of $A$ is $\begin{bmatrix}a&\\&C\end{bmatrix}$ where $C$ is the companion matrix to an irreducible polynomial of degree two over $\F_q$. We call them LP matrices because the most important invaraint subspaces for this matrix are the eigenline and an invariant plane. By Proposition~\ref{prop:ConnObstDiag2} or Proposition~\ref{prop:ConnObstIrred2}, when $q=2$, images of these matrices form isolated connected components in the reduced power graphs of $\PGL_3(\F_q)$, and each component has diameter at most $1$.
\item (Jordan pivot matrices) The canonical form of $A$ is $\begin{bmatrix}a&&1\\&a&\\&&a\end{bmatrix}$.
\item (LLP matrices)The canonical form of $A$ is $\begin{bmatrix}a&&\\&b&1\\&&b\end{bmatrix}$. (Only possible if $q\neq 2$.) We call them LLP matrices because the most important invaraint subspaces for this matrix are the two eigenlines and an invariant plane.
\item (NPJ matrices) The canonical form of $A$ is $\begin{bmatrix}a&1&\\&a&1\\&&a\end{bmatrix}$. We call them NPJ matrices because they are non-pivot Jordan matrices. When $q$ is odd, by Proposition~\ref{prop:ConnObstJordan}, images of these matrices form isolated connected components in the reduced power graphs of $\PGL_3(\F_q)$, and each component has diameter at most $1$. When $n$ is even, then the square of this matrix is a multiple of a Jordan pivot matrix. 
\end{enumerate}

\begin{defn}
For any prime power $q$, we say the \emph{Jordan type} of a matrix $A\in\GL_3(\F_q)$ is identity, pivot, LLL, irreducible, LP, Jordan pivot, LLP, or NPJ if its generalized Jordan normal form has the corresponding form classified as above.
\end{defn}

We want to analyze the pivot components, which contains pivot matrices, Jordan pivot matrices, LLL matrices, LP matrices, LLP matrices, and (when $q$ is even) NPJ matrices. Now, by calculating potential powers, we can obtain the connection diagrams below. We use solid arrows to indicate that some matrices of a certain type has a power in another type. If the arrow has the word ``char $2$'', then the connection is only possible when $q$ is even.

\begin{figure}[H]
\centering
 \begin{tikzcd}[
 %column sep=huge,row sep=huge, 
ampersand replacement=\&]
%sep
\fbox{\begin{minipage}{10em}
\centering
LLL Matrices\\
$\begin{bmatrix}
a&&\\&b&\\&&c
\end{bmatrix}$
\end{minipage}} 
%sep
\arrow[r]
\&
%sep
\fbox{\begin{minipage}{10em}
\centering
Pivot Matrices\\
$\begin{bmatrix}
a&&\\&b&\\&&b
\end{bmatrix}$
\end{minipage}} 
\&
%sep
\fbox{\begin{minipage}{10em}
\centering
LP Matrices\\
$\begin{bmatrix}
a&\\&C
\end{bmatrix}$
\end{minipage}} 
%sep
\arrow[l]
\\
\&\&
\\
%sep
\fbox{\begin{minipage}{10em}
\centering
LLP Matrices\\
$\begin{bmatrix}
a&&\\&b&1\\&&b
\end{bmatrix}$
\end{minipage}} 
%sep
\arrow[uur]
\arrow[r]
\&
%sep
\fbox{\begin{minipage}{10em}
\centering
Jordan Pivot Matrices\\
$\begin{bmatrix}
a&&\\&a&1\\&&a
\end{bmatrix}$
\end{minipage}}
%sep
%\arrow[uu,dotted,"\leq 3"']
\&
%sep
\fbox{\begin{minipage}{10em}
\centering
NPJ Matrices\\
$\begin{bmatrix}
a&1&\\&a&1\\&&a
\end{bmatrix}$
\end{minipage}}
%sep
\arrow[l,"\text{char 2}"]
%\arrow[uul,dotted,"\text{char 2}","\leq 3"']
%sep
\end{tikzcd}
\caption{When $q-1$ is not a prime power}
\label{Fig:n3q-1NotPowerJordanConn}
\end{figure}
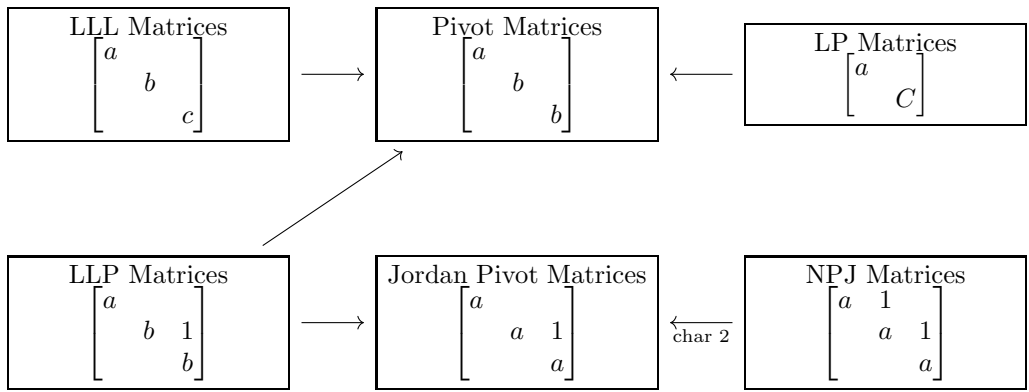

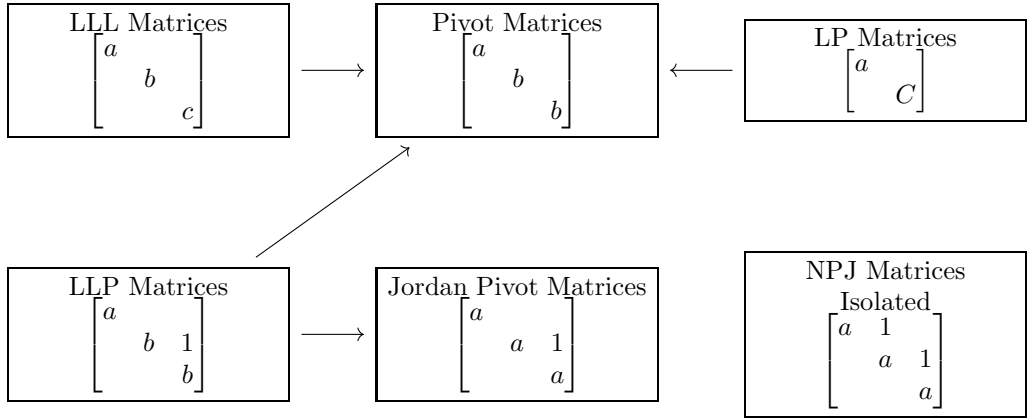
\begin{figure}[H]
\centering
 \begin{tikzcd}[
 %column sep=huge,row sep=huge, 
ampersand replacement=\&]
%sep
\fbox{\begin{minipage}{10em}
\centering
LLL Matrices\\
$\begin{bmatrix}
a&&\\&b&\\&&c
\end{bmatrix}$
\end{minipage}} 
%sep
\arrow[r]
\&
%sep
\fbox{\begin{minipage}{10em}
\centering
Pivot Matrices\\
$\begin{bmatrix}
a&&\\&b&\\&&b
\end{bmatrix}$
\end{minipage}} 
\&
%sep
\fbox{\begin{minipage}{10em}
\centering
LP Matrices\\
$\begin{bmatrix}
a&\\&C
\end{bmatrix}$
\end{minipage}} 
%sep
\arrow[l]
\\
\&\&
\\
%sep
\fbox{\begin{minipage}{10em}
\centering
LLP Matrices\\
$\begin{bmatrix}
a&&\\&b&1\\&&b
\end{bmatrix}$
\end{minipage}} 
%sep
\arrow[uur]
\arrow[r]
\&
%sep
\fbox{\begin{minipage}{10em}
\centering
Jordan Pivot Matrices\\
$\begin{bmatrix}
a&&\\&a&1\\&&a
\end{bmatrix}$
\end{minipage}}
%sep
%\arrow[uu,dotted,"\leq 3"']
\&
%sep
\fbox{\begin{minipage}{10em}
\centering
NPJ Matrices\\
Isolated\\
$\begin{bmatrix}
a&1&\\&a&1\\&&a
\end{bmatrix}$
\end{minipage}}
%sep
\end{tikzcd}
\caption{When $q=9$ or $q$ is a Fermat prime, $q\neq 3$}
\label{Fig:n3q-1OddPowerJordanConn}
\end{figure}

\begin{figure}[H]
\centering
 \begin{tikzcd}[
 %column sep=huge,row sep=huge, 
ampersand replacement=\&]
%sep
\fbox{\begin{minipage}{10em}
\centering
LLL Matrices\\
Isolated\\
$\begin{bmatrix}
a&&\\&b&\\&&c
\end{bmatrix}$
\end{minipage}} 
%sep
\&
%sep
\fbox{\begin{minipage}{10em}
\centering
Pivot Matrices\\
$\begin{bmatrix}
a&&\\&b&\\&&b
\end{bmatrix}$
\end{minipage}} 
\&
%sep
\fbox{\begin{minipage}{10em}
\centering
LP Matrices\\
$\begin{bmatrix}
a&\\&C
\end{bmatrix}$
\end{minipage}} 
%sep
\arrow[l]
\\
\&\&
\\
%sep
\fbox{\begin{minipage}{10em}
\centering
LLP Matrices\\
$\begin{bmatrix}
a&&\\&b&1\\&&b
\end{bmatrix}$
\end{minipage}} 
%sep
\arrow[uur]
\arrow[r]
\&
%sep
\fbox{\begin{minipage}{10em}
\centering
Jordan Pivot Matrices\\
$\begin{bmatrix}
a&&\\&a&1\\&&a
\end{bmatrix}$
\end{minipage}}
%sep
%\arrow[l]
%\arrow[uu,dotted,"\leq 3"']
\&
%sep
\fbox{\begin{minipage}{10em}
\centering
NPJ Matrices\\
$\begin{bmatrix}
a&1&\\&a&1\\&&a
\end{bmatrix}$
\end{minipage}}
%sep
\arrow[l]
%\arrow[uul,dotted,"\leq 3"']
%sep
\end{tikzcd}
\caption{When $q$ even and $q-1$ prime}
\label{Fig:n3q-1EvenPowerJordanConn}
\end{figure}
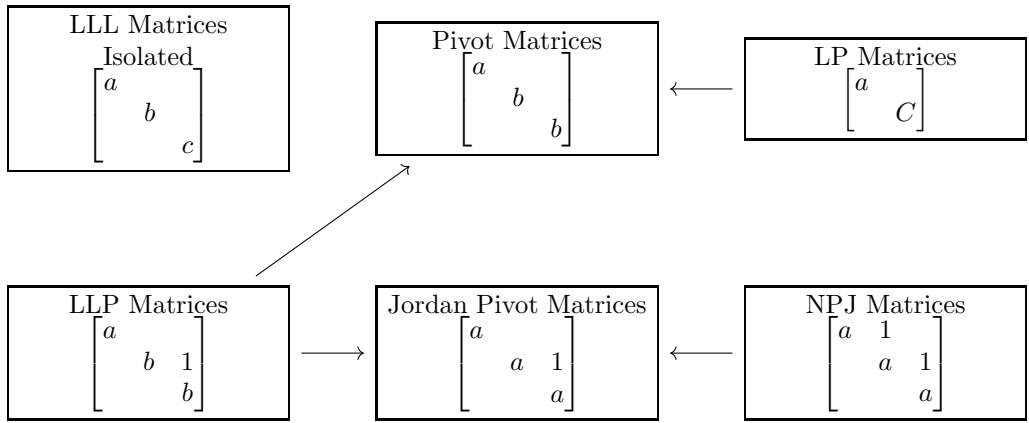

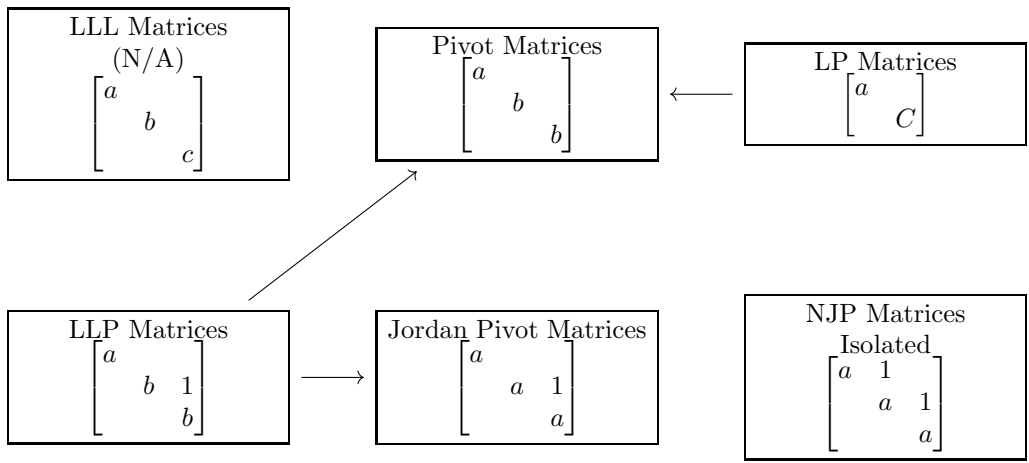
\begin{figure}[H]
\centering
 \begin{tikzcd}[
 %column sep=huge,row sep=huge, 
ampersand replacement=\&]
%sep
\fbox{\begin{minipage}{10em}
\centering
LLL Matrices\\
(N/A)\\
$\begin{bmatrix}
a&&\\&b&\\&&c
\end{bmatrix}$
\end{minipage}} 
%sep
\&
%sep
\fbox{\begin{minipage}{10em}
\centering
Pivot Matrices\\
$\begin{bmatrix}
a&&\\&b&\\&&b
\end{bmatrix}$
\end{minipage}} 
\&
%sep
\fbox{\begin{minipage}{10em}
\centering
LP Matrices\\
$\begin{bmatrix}
a&\\&C
\end{bmatrix}$
\end{minipage}} 
%sep
\arrow[l]
\\
\&\&
\\
%sep
\fbox{\begin{minipage}{10em}
\centering
LLP Matrices\\
$\begin{bmatrix}
a&&\\&b&1\\&&b
\end{bmatrix}$
\end{minipage}} 
%sep
\arrow[uur]
\arrow[r]
\&
%sep
\fbox{\begin{minipage}{10em}
\centering
Jordan Pivot Matrices\\
$\begin{bmatrix}
a&&\\&a&1\\&&a
\end{bmatrix}$
\end{minipage}}
%sep
%\arrow[l]
%\arrow[uu,dotted,"\leq 3"']
\&
%sep
\fbox{\begin{minipage}{10em}
\centering
NJP Matrices\\
Isolated\\
$\begin{bmatrix}
a&1&\\&a&1\\&&a
\end{bmatrix}$
\end{minipage}}
%sep
\end{tikzcd}
\caption{When $q=3$}
\label{Fig:n3q=3JordanConn}
\end{figure}

\begin{figure}[H]
\centering
 \begin{tikzcd}[
 %column sep=huge,row sep=huge, 
ampersand replacement=\&]
%sep
\fbox{\begin{minipage}{10em}
\centering
LLL Matrices\\
(N/A)\\
$\begin{bmatrix}
a&&\\&b&\\&&c
\end{bmatrix}$
\end{minipage}} 
%sep
\&
%sep
\fbox{\begin{minipage}{10em}
\centering
Pivot Matrices\\
(N/A)\\
$\begin{bmatrix}
a&&\\&b&\\&&b
\end{bmatrix}$
\end{minipage}} 
\&
%sep
\fbox{\begin{minipage}{10em}
\centering
LP Matrices\\
Isolated\\
$\begin{bmatrix}
a&\\&C
\end{bmatrix}$
\end{minipage}} 
%sep
\\
\&\&
\\
%sep
\fbox{\begin{minipage}{10em}
\centering
LLP Matrices\\
(N/A)\\
$\begin{bmatrix}
a&&\\&b&1\\&&b
\end{bmatrix}$
\end{minipage}} 
%sep
\&
%sep
\fbox{\begin{minipage}{10em}
\centering
Jordan Pivot Matrices\\
$\begin{bmatrix}
a&&\\&a&1\\&&a
\end{bmatrix}$
\end{minipage}}
\&
%sep
\fbox{\begin{minipage}{10em}
\centering
NPJ Matrices\\
$\begin{bmatrix}
a&1&\\&a&1\\&&a
\end{bmatrix}$
\end{minipage}}
%sep
\arrow[l]
%sep
\end{tikzcd}
\caption{When $q=2$}
\label{Fig:n3q=2JordanConn}
\end{figure}
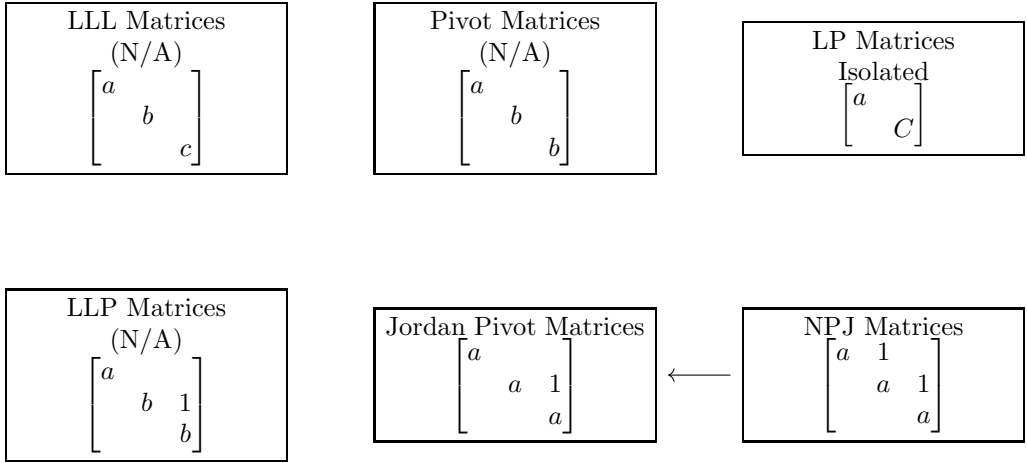

Going through these diagram, we can clearly see the following result.

\begin{lem}
\label{Lem:n3ClassifyPivotComponent}
Let $q\neq 2$ be a prime power. For any matrix $A$ whose image is in the pivot component of the reduced power graph of $\PGL_3(\F_q)$, either $A$ is similar to $\begin{bmatrix}a&\\&C\end{bmatrix}$ for some $a\in\F_q^*$ and invertible $2\times 2$ matrix $C$, or $q$ is even and $A$ is similar to $\begin{bmatrix}a&1&\\&a&1\\&&a\end{bmatrix}$ for some $a\in\F_q^*$.
\end{lem}
\begin{proof}
Check the diagrams above one by one.
\end{proof}

For this purpose, we make the following definition.

\begin{defn}
For any prime power $q$, if $A\in\GL_3(\F_q)$ is similar to $\begin{bmatrix}a&\\&C\end{bmatrix}$ for some $a\in\F_q^*$ and invertible $2\times 2$ matrix $C$, then we say $A$ is \emph{decomposable}.
\end{defn}

Note that decomposable matrices are exactly the LLL matirces, pivot matrices, LP matrices, LLP matrices,  and Jordan pivot matrices.

\section{Decomposable Matrices to Pivot Matrices}

\begin{lem}
\label{Lem:n3qPrimeAvoidADecomposable}
Let $q\neq 2$ be a power of a prime $p$, and pick any factor $p_0$ of $q-1$. Suppose a decomposable matrix $A\in\GL_3(\F_q)-Z(\GL_n(\F_q))$ has projective order coprime to $p_0$. Then a power of a root of a scalar multiple of $A$ is a pivot matrix with multiplicative order $p_0$.
\end{lem}
\begin{proof}
Suppose $A$ has projective order $k$. If $A$ is similar to $\begin{bmatrix}a&\\&C\end{bmatrix}$ for some $a\in\F_q^*$ and invertible $2\times 2$ matrix $C$, consider $a^{-1}A$, which is similar to $\begin{bmatrix}1&\\&a^{-1}C\end{bmatrix}$. Its multiplicative order must be $k$. So $a^{-1}C$ has multiplicative order $k$.

Since $k$ is coprime to $p_0$, by Corollary~\ref{Cor:IrredJordanCoprimeRoot}, $a^{-1}C$ has a $p_0$-th root $C'$ whose multiplicative order is a multiple of $p_0$. Let $x\in\F_q^*$ be any element with multiplicative order $p_0$. Then since $(C')^{kp_0}$ must be identity and $p_0$ is coprime to $p$, therefore $(C')^k$ must be diagonalizable with eigenvalue $x^{s_1},x^{s_2}$ for some positive integer $s_1,s_2$. Note that $k$ is coprime to $p_0$. So we can find a positive integer $k'$ such that $kk'$ is $1$ modulus $p_0$. 

Suppose $s_1-s_2$ is a multiple of $p_0$. So since $a^{-1}A$ is similar to $\begin{bmatrix}1&\\&a^{-1}C\end{bmatrix}$, therefore $a^{-1}A$ has a $p_0$-th root similar to $\begin{bmatrix}1&\\&x^{k'(1-s_1)}C'\end{bmatrix}$, whose $k$-th power is similar to $\begin{bmatrix}1&&\\&x&\\&&x\end{bmatrix}$, a pivot matrix with multiplicative order $p_0$.

Suppose $s_1-s_2$ is not a multiple of $p_0$. So since $a^{-1}A$ is similar to $\begin{bmatrix}1&\\&a^{-1}C\end{bmatrix}$, therefore $a^{-1}A$ has a $p_0$-th root similar to $\begin{bmatrix}1&\\&x^{-k's_1}C'\end{bmatrix}$, whose $k$-th power is similar to $\begin{bmatrix}1&&\\&1&\\&&x^{s_2-s_1}\end{bmatrix}$, a pivot matrix with multiplicative order $p_0$.
\end{proof}

\begin{lem}
\label{Lem:n3qOddAEven}
If $q$ is a power of an odd prime $p$. Suppose a decomposable matrix $A\in\GL_3(\F_q)-Z(\GL_n(\F_q))$ has even projective order. Then a power of a scalar multiple of $A$ is a pivot matrix with multiplicative order $2$.
\end{lem}
\begin{proof}
Suppose $A$ has projective order $k$.  If $A$ is similar to $\begin{bmatrix}a&\\&C\end{bmatrix}$ for some $a\in\F_q^*$ and invertible $2\times 2$ matrix $C$, consider $a^{-1}A$, which is similar to $\begin{bmatrix}1&\\&a^{-1}C\end{bmatrix}$. Its multiplicative order must be $k$. So $a^{-1}C$ has multiplicative order $k$.

Now $k$ is even. Therefore, consider $(a^{-1}A)^{\frac{k}{2}}$, which has multiplicative order $2$, and hence it must be diagonalizable over $\F_q$, and it must have eigenvalues $1,\pm 1,\pm 1$. Since it cannot be identity, $A^{\frac{k}{2}}$ is similar to $\begin{bmatrix}1&&\\&-1&\\&&-1\end{bmatrix}$ or $\begin{bmatrix}1&&\\&1&\\&&-1\end{bmatrix}$. So we are done.
\end{proof}

\begin{cor}
\label{Cor:n3qOddPath2Pivot}
If $q$ is a power of an odd prime $p$. Then for any matrix $A\in\GL_3(\F_q)-Z(\GL_n(\F_q))$, in the reduced power graph of $\PGL_3(\F_q)$, either the image of $A$ has distance at most $2$ to the image of a pivot matrix with multiplicative order $2$, or the image of $A$ has no path to the image of any pivot matrix.
\end{cor}
\begin{proof}
Note that the image of $A$ has a path to the image of any pivot matrix if and only if $A$ is not as described by Proposition~\ref{prop:ConnObstJordan} or Proposition~\ref{prop:ConnObstIrred}. By going through all possible generalized Jordan canonical form, $A$ is not as described by Proposition~\ref{prop:ConnObstJordan} or Proposition~\ref{prop:ConnObstIrred} if and only if $A$ is decomposable.

If $A$ has even projective order, then this distance is at most $1$ by Lemma~\ref{Lem:n3qOddAEven}. If $A$ has odd projective order, then this distance is at most $2$ by Lemma~\ref{Lem:n3qPrimeAvoidADecomposable}.
\end{proof}

\begin{lem}
\label{Lem:n3qEvenANotPower}
If $q\neq 2$ is a power of $2$. Fix any prime factor $p_0$ of $q-1$. Suppose the projective order of a decomposable matrix $A\in\GL_3(\F_q)-Z(\GL_n(\F_q))$ is not a power of $p_0$, then in the reduced power graph of $\PGL_3(\F_q)$, the image of $A$ has distance at most $3$ to the image of a pivot matrix with multiplicative order $p_0$.
\end{lem}
\begin{proof}
Suppose $A$ has projective order $k$. Since $k$ is not a power of $p_0$, we can find a prime factor $p_1$ of $k$ distinct from $p_0$.  Then $A'=A^{\frac{k}{p_1}}$ will have projective order $p_1$. Also note that $A'$ is a power of a decomposable matrix, hence it is decomposable itself. By Lemma~\ref{Lem:n3qPrimeAvoidADecomposable}, in the reduced power graph of $\PGL_3(\F_q)$, the image of $A'$ has distance at most $2$ to the image of a pivot matrix with multiplicative order $p_0$. So the image of $A$ has distance at most $3$ to the image of a pivot matrix with multiplicative order $p_0$.
\end{proof}

\begin{lem}
\label{Lem:n3qEvenANPJ}
If $q\neq 2$ is a power of $2$. Fix any prime factor $p_0$ of $q-1$. Suppose a matrix $A\in\GL_3(\F_q)-Z(\GL_n(\F_q))$ is similar to $\begin{bmatrix}a&1&\\&a&1\\&&a\end{bmatrix}$ for some $a\in\F_q^*$, then in the reduced power graph of $\PGL_3(\F_q)$, the image of $A$ has distance at most $3$ to the image of a pivot matrix with multiplicative order $p_0$.
\end{lem}
\begin{proof}
Note that $A'=(a^{-1}A)^2$ is similar to $\begin{bmatrix}1&&1\\&1&\\&&1\end{bmatrix}$, which is a decomposable matrix with multiplicative order $2$. By Lemma~\ref{Lem:n3qPrimeAvoidADecomposable}, in the reduced power graph of $\PGL_3(\F_q)$, the image of $A'$ has distance at most $2$ to the image of a pivot matrix with multiplicative order $p_0$. So the image of $A$ has distance at most $3$ to the image of a pivot matrix with multiplicative order $p_0$.
\end{proof}

\begin{lem}
\label{Lem:n3qEvenNotMersenneAPower}
If $q\neq 2$ is a power of $2$ and $q-1$ is not prime. Fix any prime factor $p_0$ of $q-1$. Suppose the projective order of a decomposable matrix $A\in\GL_3(\F_q)-Z(\GL_n(\F_q))$ is a power of $p_0$, then in the reduced power graph of $\PGL_3(\F_q)$, the image of $A$ has distance at most $4$ to the image of a pivot matrix with multiplicative order $p_0$.
\end{lem}
\begin{proof}
Since $q$ is even and $q-1$ is not prime, by Lemma~\ref{Lem:ConseqPP}, $q-1$ is not a prime power. So we can find a prime factor $p_1$ of $q-1$ distinct from $p_0$.

If the projective order of $A$ is a power of $p_0$, then it is coprime to $p_1$. By Lemma~\ref{Lem:n3qPrimeAvoidADecomposable}, in the reduced power graph of $\PGL_3(\F_q)$, the image of $A$ has distance at most $2$ to the image of a pivot matrix with multiplicative order $p_1$. However, this pivot matrix has multiplicative order coprime to $p_0$. So again by Lemma~\ref{Lem:n3qPrimeAvoidADecomposable}, the image of this pivot matrix has distance at most $2$ to the image of a pivot matrix with multiplicative order $p_0$. So all in all, the image of $A$ has distance at most $4$ to the image of a pivot matrix with multiplicative order $p_0$.
\end{proof}

\begin{lem}
\label{Lem:n3qEvenMersenneAPower}
If $q\neq 2$ is a power of $2$ and $p_0=q-1$ is a prime.Suppose the projective order of a matrix $A\in\GL_3(\F_q)-Z(\GL_n(\F_q))$ is a power of $p_0$, then $A$ is diagonalizable over $\F_q$ or its characteristic polynomial is irreducible. In particular, either $A$ is a pivot matrix itself, or $A$ is as described in Proposition~\ref{prop:ConnObstDiag} or Proposition~\ref{prop:ConnObstIrred}.
\end{lem}
\begin{proof}
Since the projective order of $A$ is a power of $p_0$, which is odd, by Lemma~\ref{Lem:NonJordayTypeClassify}, the generalized Jordan canonical form of $A$ is block diagonal with companion matrices of irreducible polynomials as diagonal blocks. 

If $C$ is one of these blocks and it is $m\times m$ for some $m>1$, then by Corollary~\ref{Cor:CompanionPower}, its multiplicative order must divide $q^m-1$, but cannot divide $q-1$. Hence if the projective order of a matrix $A\in\GL_3(\F_q)-Z(\GL_n(\F_q))$ is a power of $p_0$, then $q^m-1$ must be a multiple of $p_0^2$. Since $q$ is $1$ modulus the odd prime $p_0$, clearly $q+1$ is cannot be a multiple of $p_0$. So $q^2-1$ is not a multiple of $p_0^2$. So $m\neq 2$. Hence all diagonal blocks in the generalized Jordan canonical form of $A$ must be $1\times 1$ or $3\times 3$. If one block is $3\times 3$, then $A$ has irreducible polynomial. If all blocks are $1\times 1$, then $A$ is diagonalizable over $\F_q$.
\end{proof}

\section{Upper Bounds when $q-1$ is not a prime power}

Let $q\neq 2$ be a power of a prime $p$. Suppose $q-1$ is not a prime power. Fix any prime factor $p_0$ of $q-1$. Our goal here is to find a short path (distance at most $4$) in the projectively reduced power graph of $\GL_3(\F_q)$ between any two pivot matrices whose multiplicative order is $p_0$.

Pick any $b\in\F_q$ such that $x\neq 0,1$. In the group $\GL_3(\F_q)$, let $S_1,S_2,S_3$ be the centralizers of $\begin{bmatrix}1&&\\&b&\\&&b\end{bmatrix},\begin{bmatrix}b&&\\&1&\\&&b\end{bmatrix},\begin{bmatrix}b&&\\&b&\\&&1\end{bmatrix}$ respectively. Let $S_J$ be the centralizer of $\begin{bmatrix}1&&\\&1&1\\&&1\end{bmatrix}$. Finally, let $S=S_2\cup S_3\cup S_J$.

\begin{lem}
\label{Lem:n3q-1NotPowerCentralGenerate}
Any $X\in\GL_3(\F_q)$, then $X=X_1X_2X_3$ for some matrices $X_1,X_3\in S_1$ and $X_2\in S$.
\end{lem}
\begin{proof}
Suppose $X=\begin{bmatrix}a&\bs v^{\T}\\ \bs w& B\end{bmatrix}$ for $\bs v,\bs w\in\F_q^2$ and a $2\times 2$ matrix $B$ over $\F_q$.

Suppose $B$ is not invertible. Since $X$ is invertible, we cannot have $B=\begin{bmatrix}0&0\\0&0\end{bmatrix}$. Hence $B$ has rank $1$, and we can find $B=B_1\begin{bmatrix}1&0\\0&0\end{bmatrix}B_2$ for some invertible $B_1,B_2\in\GL_2(\F_q)$. Then 
\[
X=\begin{bmatrix}1&\\&B_1\end{bmatrix}\begin{bmatrix}a&x&y\\c&1&0\\d&0&0\end{bmatrix}\begin{bmatrix}1&\\&B_2\end{bmatrix}.
\]

Since $X$ is invertible, we must have $d,y\neq 0$. So we have
\[
X=\begin{bmatrix}1&\\&B_1\end{bmatrix}
\begin{bmatrix}1&&\\&1&\frac{c}{d}\\&&1\end{bmatrix}
\begin{bmatrix}a&0&y\\0&1&0\\d&0&0\end{bmatrix}
\begin{bmatrix}1&&\\&1&\\&\frac{x}{y}&1\end{bmatrix}
\begin{bmatrix}1&\\&B_2\end{bmatrix}\in S_1S_2S_1\subseteq S_1SS_1.
\]

Suppose $\bs v=\bs w=\bs 0$. Then $X\in S_1\subseteq S_1SS_1$.

Suppose $\bs v=\bs 0$ but $\bs w\neq\bs 0$. Then $B$ must be invertible. Pick $Y\in\GL_2(\F_q)$ such that $BY^{-1}\begin{bmatrix}1\\0\end{bmatrix}=\bs w$. Then
\[
X=\begin{bmatrix}1&\\&BY^{-1}\end{bmatrix}
\begin{bmatrix}a&&\\1&1&\\&&1\end{bmatrix}\begin{bmatrix}1&\\&Y\end{bmatrix}\in S_1S_3S_1\subseteq S_1SS_1.
\]

Similarly, if $\bs w=\bs 0$ but $\bs v\neq \bs 0$, then again $X\in S_1S_3S_1\subseteq S_1SS_1$.

From now on, we assume that $B$ is invertible and $\bs v,\bs w\neq\bs 0$. Pick $Y\in\GL_2(\F_q)$ such that $BY^{-1}\begin{bmatrix}1\\0\end{bmatrix}=\bs w$. Suppose $\bs v^{\T}Y^{-1}=\begin{bmatrix}x&y\end{bmatrix}$. Then
\[
X=\begin{bmatrix}1&\\&BY^{-1}\end{bmatrix}
\begin{bmatrix}a&x&y\\1&1&\\&&1\end{bmatrix}\begin{bmatrix}1&\\&Y\end{bmatrix}.
\]

Here, $x$ is actually determined by $\bs v,\bs w,B$ alone. Indeed we have
\[
x=\begin{bmatrix}x&y\end{bmatrix}\begin{bmatrix}1\\0\end{bmatrix}=\bs v^{\T}Y^{-1}YB^{-1}\bs w=\bs v^{\T}B^{-1}\bs w.
\]

If $\bs v^{\T}B^{-1}\bs w\neq 0$, then we further have
\[
X=\begin{bmatrix}1&\\&BY^{-1}\end{bmatrix}
\begin{bmatrix}1&&\\&1&-\frac{y}{x}\\&&1\end{bmatrix}
\begin{bmatrix}a&x&0\\1&1&\\&&1\end{bmatrix}
\begin{bmatrix}1&&\\&1&\frac{y}{x}\\&&1\end{bmatrix}
\begin{bmatrix}1&\\&Y\end{bmatrix}\in S_1S_3S_1\subseteq S_1SS_1.
\]

Now suppose $\bs v^{\T}B^{-1}\bs w=0$. Then we have
\[
X=\begin{bmatrix}1&\\&BY^{-1}\end{bmatrix}
\begin{bmatrix}a&0&y\\1&1&\\&&1\end{bmatrix}\begin{bmatrix}1&\\&Y\end{bmatrix}\in S_1S_JS_1\subseteq S_1SS_1.
\]
\end{proof}

\begin{cor}
\label{Cor:n3qNotPowerPivotDistance}
Let $q\neq 2$ be any prime power such that $q-1$ is not a prime power. Let $p_0$ be any prime factor of $q-1$. Then in the projectively reduced power graph of $\GL_3(\F_q)$, any two pivot matrices with multiplicative order $p_0$ has distance at most $4$.
\end{cor}
\begin{proof}
Suppose $A,B$ are two pivot matrices in $\GL_3(\F_q)$ with multiplicative order $p_0$. Say $A=X_A\begin{bmatrix}a_1&&\\&a_2&\\&&a_2\end{bmatrix}X_A^{-1}$ and $B=X_B\begin{bmatrix}b_1&&\\&b_2&\\&&b_2\end{bmatrix}X_B^{-1}$. Let $X=X_A^{-1}X_B$. By Lemma~\ref{Lem:n3q-1NotPowerCentralGenerate}, $X=X_1X_2X_3$ for some matrices $X_1,X_3\in S_1$ and $X_2\in S$.

Note that $\begin{bmatrix}a_1&&\\&a_2&\\&&a_2\end{bmatrix},\begin{bmatrix}b_1&&\\&b_2&\\&&b_2\end{bmatrix}$ both have centralizer $S_1$. Since $q-1$ is not a prime power, we can find a prime factor $p_1$ of $q-1$ different from $p_0$. Let $x\in\F_q^*$ be any element with multiplicative order $p_1$. If $X_2\in S_2$, $S_3$ or $S_J$, then we set $D=\begin{bmatrix}x&&\\&1&\\&&x\end{bmatrix}$, $\begin{bmatrix}x&&\\&x&\\&&1\end{bmatrix}$ or $\begin{bmatrix}1&&\\&1&1\\&&1\end{bmatrix}$ respectively. 

Then $D$ has multiplicative order $p_1$ or $p$, which is coprime to $p_0$ either way. Furthermore, we always have $\begin{bmatrix}a_1&&\\&a_2&\\&&a_2\end{bmatrix}D=D\begin{bmatrix}a_1&&\\&a_2&\\&&a_2\end{bmatrix}$ and $\begin{bmatrix}b_1&&\\&b_2&\\&&b_2\end{bmatrix}D=D\begin{bmatrix}b_1&&\\&b_2&\\&&b_2\end{bmatrix}$. Consequently, $\begin{bmatrix}a_1&&\\&a_2&\\&&a_2\end{bmatrix}$ and $D$ are both powers of $\begin{bmatrix}a_1&&\\&a_2&\\&&a_2\end{bmatrix}D$, and in the same manner, $\begin{bmatrix}b_1&&\\&b_2&\\&&b_2\end{bmatrix}$ and $D$ are both powers of $\begin{bmatrix}b_1&&\\&b_2&\\&&b_2\end{bmatrix}D$. Therefore, in the projectively reduced power graph of $\GL_3(\F_q)$, $\begin{bmatrix}a_1&&\\&a_2&\\&&a_2\end{bmatrix}$ and $\begin{bmatrix}b_1&&\\&b_2&\\&&b_2\end{bmatrix}$ both have distance at most $2$ to $D$.

So we have path
\begin{align*}
&A=X_A\begin{bmatrix}a_1&&\\&a_2&\\&&a_2\end{bmatrix}X_A^{-1}=X_AX_1\begin{bmatrix}a_1&&\\&a_2&\\&&a_2\end{bmatrix}X_1^{-1}X_A^{-1}\\
\to &X_AX_1DX_1^{-1}X_A^{-1}=X_AX_1X_2DX_2^{-1}X_1^{-1}X_A^{-1}\\
\to &X_AX_1X_2\begin{bmatrix}b_1&&\\&b_2&\\&&b_2\end{bmatrix}X_2^{-1}X_1^{-1}X_A^{-1}\\
=&X_AX_1X_2X_3\begin{bmatrix}b_1&&\\&b_2&\\&&b_2\end{bmatrix}X_3^{-1}X_2^{-1}X_1^{-1}X_A^{-1}\\
=&X_AX\begin{bmatrix}b_1&&\\&b_2&\\&&b_2\end{bmatrix}X^{-1}X_A^{-1}=X_B\begin{bmatrix}b_1&&\\&b_2&\\&&b_2\end{bmatrix}X_B^{-1}=B.
\end{align*}

Here each arrow means distance at most $2$. So $A,B$ have distance at most $4$ between them.
\end{proof}

\begin{prop}
If $q$ is an odd prime power, and $q-1$ is not a prime power, then in the reduce power graph of $\PGL_3(\F_q)$, the pivot component has diameter at most $8$.
\end{prop}
\begin{proof}
Suppose $A,B$ are any two matrices in the pivot component. By Figure~\ref{Fig:n3q-1NotPowerJordanConn}, all matrices in the pivot component are decomposable. Since $q$ is odd, by Corollary~\ref{Cor:n3qOddPath2Pivot}, in the reduce power graph of $\PGL_3(\F_q)$, images $A,B$ will respectively have distance at most $2$ to images of some pivot matrices $A',B'$ with multiplicative order $2$. Then by Corollary~\ref{Cor:n3qNotPowerPivotDistance}, images of $A',B'$ have distance at most $4$ between them. So in total, $A,B$ have distance at most $2+2+4=8$ between them.
\end{proof}

\begin{prop}
If $q$ is power of $2$, and $q-1$ is not a prime, then in the reduce power graph of $\PGL_3(\F_q)$, the pivot component has diameter at most $10$.
\end{prop}
\begin{proof}
Suppose $A,B$ are any two matrices in the pivot component. Then by Lemma~\ref{Lem:n3ClassifyPivotComponent}, each of them is either decomposable, or an NPJ matrix. 

Let $p_0,p_1$ be any two distinct prime factors of $q-1$. Suppose $A,B$ are both not a decomposable matrix whose projective order is a power of $p_i$. Then by Lemma~\ref{Lem:n3qEvenANotPower} and Lemma~\ref{Lem:n3qEvenANPJ}, in the reduce power graph of $\PGL_3(\F_q)$, the images of $A$ and $B$ will respectively have distance at most $3$ to the images of some pivot matrices $A'$ and $B'$ with multiplicative order $p_i$. Then by Corollary~\ref{Cor:n3qNotPowerPivotDistance}, the distance between images of $A',B'$ is at most $4$. Therefore the total distance between images of $A,B$ is at most $3+3+4=10$.

If the above assumption does not happen, then one of $A,B$ is a decomposable matrix whose projective order is a power of $p_0$, and the other is a decomposable matrix whose projective order is a power of $p_1$. Say $A$ has projective order a power of $p_0$ and $B$ has projective order a power of $p_1$. Then by Lemma~\ref{Lem:n3qEvenNotMersenneAPower}, the image of $A$ has distance at most $4$ to the image of a pivot matrix $A'$ with multiplicative order $p_0$. By Lemma~\ref{Lem:n3qPrimeAvoidADecomposable}, the image of $B$ has distance at most $2$ to the image of a pivot matrix $B'$ with multiplicative order $p_0$. Then by Corollary~\ref{Cor:n3qNotPowerPivotDistance}, the distance between images of $A',B'$ is at most $4$. Therefore the total distance between images of $A,B$ is at most $4+2+4=10$.
\end{proof}

\section{Upper bound when $q-1$ is a prime power}

\begin{prop}
Let $q$ be a Fermat prime or $q=9$. Then the reduced power graph of $\PGL_3(\F_q)$ has diameter at most $12$.
\end{prop}
\begin{proof}
Suppose $A,B$ are any two matrices in the pivot component. By Figure~\ref{Fig:n3q-1OddPowerJordanConn} and Figure~\ref{Fig:n3q=3JordanConn}, all matrices in the pivot component are decomposable. Since $q$ is odd, by Corollary~\ref{Cor:n3qOddPath2Pivot}, in the reduce power graph of $\PGL_3(\F_q)$, images $A,B$ will respectively have distance at most $2$ to images of some pivot matrices $A',B'$ with multiplicative order $2$. Then by Proposition~\ref{Prop:Pivot3}, images of $A',B'$ have distance at most $8$ between them. So in total, $A,B$ have distance at most $2+2+8=12$ between them.
\end{proof}

\begin{prop}
Let $q\neq 2$ be a power of $2$ such that $q-1$ is prime. Then the reduced power graph of $\PGL_3(\F_q)$ has diameter at most $13$.
\end{prop}
\begin{proof}
Suppose $A,B$ are any two matrices in the pivot component. By Figure~\ref{Fig:n3q-1EvenPowerJordanConn}, all matrices in the pivot component are decomposable or NJP matrices, and they cannot be LLL matrices. 

Suppose $A$ is not an NJP matrix. By Lemma~\ref{Lem:n3qEvenMersenneAPower}, either $A$ is a pivot matrix, or the projective order of $A$ is not a power of the prime $q-1$. Then by Lemma~\ref{Lem:n3qEvenANotPower}, the $A$ has distance at most $2$ to the image of a pivot matrix $A'$.

Now, if $B$ is also not an NPJ matrix, then similarly the image of $B$ has distance at most $2$ to the image of a pivot matrix $B'$. If $B$ is an NPJ matrix, by Lemma~\ref{Lem:n3qEvenANPJ}, the image of $B$ has distance at most $3$ to the image of a pivot matrix $B'$. 

By Proposition~\ref{Prop:Pivot3}, images of $A',B'$ have distance at most $8$ between them. So in total, images of $A,B$ have distance at most $2+3+8=13$ between them.

Similarly, if $B$ is not an NPJ matrix, images of $A,B$ have distance at most $2+3+8=13$ between them.

Finally, suppose $A,B$ are both NPJ matrices. Then $A^2$ and $B^2$ will be Jordan pivot matrices. By Proposition~\ref{Prop:PivotJordan}, images of $A^2,B^2$ have distance at most $8$ between them. So in total, images of $A,B$ have distance at most $1+1+8=10$ between them.

All in all, images of $A,B$ have distance at most $13$ between them.
\end{proof}

\begin{prop}
The reduced power graph of $\PGL_3(\F_3)$ has diameter at most $11$.
\end{prop}
\begin{proof}
Suppose $A,B$ are any two matrices in the pivot component. If both matrices are Jordan pivot matrices, by Proposition~\ref{Prop:PivotJordan}, images of $A,B$ have distance at most $8$ between them.

Now WLOG suppose $A$ is not a Jordan pivot matrix. Then by Figure~\ref{Fig:n3q=3JordanConn}, $A$ is an LP matrix or an LLP matrix. If $A$ is an LP matrix, then its multiplicative order must divides $3^2-1=8$. So it must have even projective order. By Lemma~\ref{Lem:n3qOddAEven}, it has distance at most $1$ to the image of a pivot matrix $A'$. On the other hand, if $A$ is an LLP matrix, then $A'=A^3$ is a pivot matrix. Either way, the image of a pivot matrix $A'$ has distance $1$ to the image of $A$.

By Figure~\ref{Fig:n3q=3JordanConn}, $B$ is a decomposable matrix. Since $q$ is odd, by Corollary~\ref{Cor:n3qOddPath2Pivot}, image $B$ will have distance at most $2$ to image of some pivot matrix $B'$.

By Proposition~\ref{Prop:Pivot3}, images of $A',B'$ have distance at most $8$ between them. So in total, images of $A,B$ have distance at most $1+2+8=11$ between them.
\end{proof}

\section{Connection restrictions to pivot matrices}

First, let us discuss the cases when two matrices of the same Jordan type are connected by an edge in the reduced power graph of $\PGL_3(\F_q)$.

\begin{lem}
\label{Lem:n3SameTypeEdge}
In the reduced power graph of $\PGL_3(\F_q)$, if the image of a matrix $A$ and the image of a matrix $B$ are connected by an edge, and $A,B$ have the same Jordan type, then they have identical invariant subspaces, eigenspaces and generalized eigenspaces.
\end{lem}
\begin{proof}
If the image of a matrix $A$ and the image of a matrix $B$ are connected by an edge, then WLOG say $xB=A^k$ for some positive integer $k$ and some $x\in\F_q^*$. Then all invariant subspaces of $A$ are invariant subspaces of $A^k=xB$, and thus also invariant subspaces of $B$. Similarly, all eigenspaces of $A$ are eigenspaces of $B$, and all generalized eigenspaces of $A$ are generalized eigenspaces of $B$.

Suppose $A,B$ have the same Jordan type. Then they have the same finite number of invariant subspaces, eigenspaces and generalized eigenspaces. Hence they have identical invariant subspaces, eigenspaces and generalized eigenspaces.
\end{proof}

Next, let us discuss the cases when a matrix and a pivot matrix are connected by an edge in the reduced power graph of $\PGL_3(\F_q)$.

\begin{lem}
\label{Lem:n3LLLPivotEdge}
In the reduced power graph of $\PGL_3(\F_q)$, if the image of an LLL matrix $A$ and the image of a pivot matrix $B$ are connected by an edge, then each eigenspace of $B$ is spanned by some eigenspaces of $A$.
\end{lem}
\begin{proof}
If the image of an LLL matrix $A$ and the image of a pivot matrix $B$ are connected by an edge, then $xB=A^k$ for some positive integer $k$ and some $x\in\F_q^*$. 

Since $A$ is an LLL matrix, $A=X\begin{bmatrix}a_1&&\\&a_2&\\&&a_3\end{bmatrix}X^{-1}$ for some invertible $X$ and some $a_1,a_2,a_3\in\F_q^*$. Then $B=X\begin{bmatrix}x^{-1}a_1^k&&\\&x^{-1}a_2^k&\\&&x^{-1}a_3^k\end{bmatrix}X^{-1}$. 

Since $A$ is an LLL matrix, the eigenspaces of $A$ are the three lines spanned by the three columns of $X$ respectively. Since $B$ is a pivot matrix, its eigenspaces are a line spanned by a column of $X$, and a plane spanned by the other two column of $X$. So we are done.
\end{proof}

\begin{lem}
\label{Lem:n3LPPivotEdge}
In the reduced power graph of $\PGL_3(\F_q)$, if the image of an LP matrix $A$ and the image of a pivot matrix $B$ are connected by an edge, then the invariant subspaces of $A$ are exactly the eigenspaces of $B$.
\end{lem}
\begin{proof}
If the image of an LP matrix $A$ and the image of a pivot matrix $B$ are connected by an edge, then a power of $A$ must be $xB$ for some $x\in\F_q^*$. 

Since $A$ is an LP matrix, $A=X\begin{bmatrix}a&\\&C\end{bmatrix}X^{-1}$ for some invertible $X$, some $a\in\F_q^*$ and a companion matrix $C$ to some irreducible polynomial. Then $B=X\begin{bmatrix}x^{-1}a^k&\\&x^{-1}C^k\end{bmatrix}X^{-1}$. Since $B$ is a pivot matrix, $x^{-1}C^k$ must be diagonalizable over $\F_q$ as well. But by Lemma~\ref{Lem:CompanionPower}, $C^k$ is either not diagonalizable over $\F_q$, or a multiple of identity. So we must have $C^k=\begin{bmatrix}y&\\&y\end{bmatrix}$ for some $y\in\F_q^*$, and $B=X\begin{bmatrix}x^{-1}a^k&&\\&x^{-1}y&\\&&x^{-1}y\end{bmatrix}X^{-1}$ with $x^{-1}a^k\neq x^{-1}y$.

Now the invariant subspaces of $A$ are the line spanned by the first column of $X$, and the plane spanned by the last two columns of $X$. But these are exactly the eigenspaces of $B$.
\end{proof}

\begin{lem}
\label{Lem:n3LLPPivotEdge}
In the reduced power graph of $\PGL_3(\F_q)$, if the image of an LLP matrix $A$ and the image of a pivot matrix $B$ are connected by an edge, then the generalized subspaces of $A$ are exactly the eigenspaces of $B$.
\end{lem}
\begin{proof}
If the image of an LLP matrix $A$ and the image of a pivot matrix $B$ are connected by an edge, then a power of $A$ must be $xB$ for some $x\in\F_q^*$. 

Since $A$ is an LLP matrix, $A=X\begin{bmatrix}a_1&&\\&a_2&1\\&&a_2\end{bmatrix}X^{-1}$ for some invertible $X$ and some $a_1,a_2\in\F_q^*$. Then $B=X\begin{bmatrix}x^{-1}a_1^k&&\\&x^{-1}a_2^k&x^{-1}ka_2^{k-1}\\&&x^{-1}a_2^k\end{bmatrix}X^{-1}$. Since $B$ is a pivot matrix, $\begin{bmatrix}x^{-1}a_2^k&x^{-1}ka_2^{k-1}\\&x^{-1}a_2^k\end{bmatrix}$ must be diagonalizable over $\F_q$ as well, so $x^{-1}ka_2^{k-1}=0$ and $B=X\begin{bmatrix}x^{-1}a_1^k&&\\&x^{-1}a_2^k&\\&&x^{-1}a_2^k\end{bmatrix}X^{-1}$ with $x^{-1}a_1^k\neq x^{-1}a_2^k$.

Now the generalized subspaces of $A$ are the line spanned by the first column of $X$, and the plane spanned by the last two columns of $X$. But these are exactly the eigenspaces of $B$.
\end{proof}

\begin{lem}
\label{Lem:n3PivotPivotDistance3}
In the reduced power graph of $\PGL_3(\F_q)$, if the image of a pivot matrix $A$ and the image of a pivot matrix $B$ have distance at most $3$, then either they have the same eigenspaces, or the $2$-dimensional eigenspace of $A$ contains the $1$-dimensional eigenspace of $B$ and vice versa.
\end{lem}
\begin{proof}
Suppose in the reduced power graph of $\PGL_3(\F_q)$, the image of $A$ has an edge to the image of $A'$, which has an edge to the image of $B'$, which has an edge to the image of $B'$. By going through the diagrams from Figure~\ref{Fig:n3q-1NotPowerJordanConn} to Figure~\ref{Fig:n3q=3JordanConn}, $A',B'$ must have the same Jordan type and must be LLL matrices or LP matrices or LLP matrices.

If $A',B'$ are both LP or both LLP matrices, then they have the same invariant subspaces and generalized eigenspaces by Lemma~\ref{Lem:n3SameTypeEdge}. Then by Lemma~\ref{Lem:n3LPPivotEdge} and Lemma~\ref{Lem:n3LLPPivotEdge}, $A,B$ must have the same eigenspaces.

If $A',B'$ are both LLL matrices, then they have the same eigenspaces by Lemma~\ref{Lem:n3SameTypeEdge}. Say the three $1$-dimensional eigenspaces are $W_1,W_2,W_3$. By Lemma~\ref{Lem:n3LLLPivotEdge}, eigenspaces of $A$ and of $B$ are spanned by these. By exhausting all possibilities, either $A,B$ have the same eigenspaces, or the $2$-dimensional eigenspace of $A$ contains the $1$-dimensional eigenspace of $B$ and vice versa.
\end{proof}

Here are some special restrictions when $q=9$ or $q$ is a Fermat prime.

\begin{lem}
\label{Lem:n3q9FermatLLLPivotEdge}
Suppose $q=9$ or $q$ is a Fermat prime. For an LLL matrix $A\in\GL_3(\F_q)$, let $k$ be the projective order of $A$. Then $A^{\frac{k}{2}}$ is a pivot matrix. Furthermore, in the reduced power graph of $\PGL_3(\F_q)$, suppose the image $A$ is connected to the image of a pivot matrix $B$ by an edge, then $A^{\frac{k}{2}}$ and $B$ must have identical eigenspaces.
\end{lem}
\begin{proof}
Since $A$ is an LLL matrix, $A=X\begin{bmatrix}a_1&&\\&a_2&\\&&a_3\end{bmatrix}X^{-1}$ for some invertible $X$ and some $a_1,a_2,a_3\in\F_q^*$. Since $k$ is the projective order of $A$, and $q-1$ is a power of $2$, $k$ must be a power of $2$. So $a_1^k=a_2^k=a_3^k$, but $a_1^{\frac{k}{2}},a_2^{\frac{k}{2}},a_3^{\frac{k}{2}}$ cannot all be identical, and they are all square roots of the same element in $\F_q^*$. Since each element in $\F_q^*$ has at most two square roots, $A^{\frac{k}{2}}$ cannot have three distinct eigenvalues. So it is indeed a pivot matrix.

Now, if the image of an LLL matrix $A$ and the image of a pivot matrix $B$ are connected by an edge, then $xB=A^{t}$ for some positive integer $t$ and some $x\in\F_q^*$. Since $B$ is not a scalar multiple of identity, $k$ cannot divide $t$. So let $t'$ be the power of $2$ such that $\frac{k}{2}$ divides $tt'$ but $k$ does not divide $tt'$, and set $B'=(xB)^{t'}$. Then $B'$ is a power of $A$ that is not a scalar multiple of identity, but it is a power of $A^{\frac{k}{2}}$. Hence it is a pivot matrix. So the image of $B$ and the image of $B'$ are connected by an edge, and the image of $B'$ and the image of $A^{\frac{k}{2}}$ are connected by an edge. By Lemma~\ref{Lem:n3SameTypeEdge}, they all of the same eigenspaces.
\end{proof}

\begin{lem}
\label{Lem:n3q9FermatLLLParallelEdge}
Suppose $q=9$ or $q$ is a Fermat prime. Suppose in the reduced power graph of $\PGL_3(\F_q)$, the image of an LLL matrix $A$ and the image of an LLL matrix $B$ are connected by an edge. Let $k_A,k_B$ be the projective order of $A,B$ respectively. Then the pivot matrices $A^{\frac{k_A}{2}}$ and $B^{\frac{k_B}{2}}$ must have identical eigenspaces.
\end{lem}
\begin{proof}
If the image of a matrix $A$ and the image of a matrix $B$ are connected by an edge, then WLOG say $xB=A^k$ for some positive integer $k$ and some $x\in\F_q^*$. Then $B^{\frac{k_B}{2}}$ is a scalar multiple of a power of $A$, and it is a pivot matrices. So by Lemma~\ref{Lem:n3q9FermatLLLPivotEdge}, $A^{\frac{k_A}{2}}$ and $B^{\frac{k_B}{2}}$ must have identical eigenspaces.
\end{proof}

\begin{lem}
\label{Lem:n3q9FermatPivotPivotPath}
Suppose $q\neq 2$ and $q-1$ is a prime power. In the reduced power graph of $\PGL_3(\F_q)$, suppose the image of a pivot matrix $A$ is connected to the image of a pivot matrix $B$ by an path that does not include any Jordan pivot matrix, then $A$ and $B$ must have identical eigenspaces.
\end{lem}
\begin{proof}
WLOG we can assume that there is no other pivot matrices on this path other than $A$ and $B$. By Figure~\ref{Fig:n3q-1OddPowerJordanConn} and Figure~\ref{Fig:n3q-1EvenPowerJordanConn}, then all matrices in this path between $A$ and $B$ must have the same Jordan type. They must all be LP matrices, or all be LLP matrices, or all be LLL matrices.

Suppose they are all LP matrices. Then by Lemma~\ref{Lem:n3SameTypeEdge}, since all of them are connected by edges, they all have the same invariant subspaces. By Lemma~\ref{Lem:n3LPPivotEdge}, the eigenspaces of $A$ and $B$ must be identical to these invariant subspaces. So $A$ and $B$ must have identical eigenspaces.

Suppose they are all LLP matrices. Then by Lemma~\ref{Lem:n3SameTypeEdge}, since all of them are connected by edges, they all have the same generalized eigenspaces. By Lemma~\ref{Lem:n3LLPPivotEdge}, the eigenspaces of $A$ and $B$ must be identical to these generalized eigenspaces. So $A$ and $B$ must have identical eigenspaces.

Finally, suppose they are all LLL matrices. Note that in this case, $q$ must be odd, so $q=9$ or $q$ is a Fermat prime by Lemma~\ref{Lem:ConseqPP}. Suppose the path goes from the image of $A$ to the image of $A_1$, to the image of $A_2$, \dots, to the image of $A_t$, and finally to the image of $B$. Let $k_i$ be the projective order of $A_i$. By Lemma~\ref{Lem:n3q9FermatLLLParallelEdge}, all $A_i^{\frac{k_i}{2}}$ have the same eigenspaces. And by Lemma~\ref{Lem:n3q9FermatLLLPivotEdge}, $A$ and $B$ must also have the same eigenspaces as these matrices.
\end{proof}

\section{Lower Bounds when $q-1$ is not a prime power}

Let $q\neq 2$ be a power of a prime $p$. Suppose $q-1$ is not a prime power. When $q$ is odd, let $p_0$ be any odd prime factor of $q-1$, and pick any $x\in\F_q^*$ with multiplicative order $p_0$. When $q$ is even, pick any multiplicative generator $x\in\F_q^*$, and set $p_0=q-1$. Note that either way, we must have $x\neq 0,1$, and both $x$ and $x^2$ have multiplicative order $p_0$.

Set $A=\begin{bmatrix}1&&\\&x&\\&&x^2\end{bmatrix}$. Pick any $x'\neq 0,1,-2$ in $\F_q$. By our assumption, $q>3$, so this is possible. Set $X=\begin{bmatrix}x'&1&1\\1&x'&1\\1&1&x'\end{bmatrix}$, which has determinant $(x'+2)(x'-1)^2\neq 0$, so it is invertible. Finally, set $B=XAX^{-1}$. Note that $A,B$ must be LLL matrices with projective and multiplicative order $p_0$. I claim that in the reduce power graph of $\PGL_3(\F_q)$, images of $A,B$ will have distance at least $8$ when $q$ is odd, and at least $10$ when $q$ is even.

\begin{lem}
\label{Lem:n3qNotPowerBadLLLExample}
Let $q\neq 2$ be a power of a prime $p$. Suppose $q-1$ is not a prime power. If $q$ is odd, in the reduce power graph of $\PGL_3(\F_q)$, the distance between the image of $A$ or $B$ to the image of any pivot matrix is at least $2$. If $q$ is even, this distance is at least $3$. 
\end{lem}
\begin{proof}
It is enough to prove it for $A$, as $A,B$ are similar matrices. 

Note that $A$ has eigenvalue $1,x,x^2$. So for any $y\in\F_q^*$, the matrix $yA^k$ has eigenvalues $y,yx^k,yx^{2k}$. But since $x$ has odd multiplicitive order $p_0$, $yA^k$ has all three eigenvalues distinct when $k$ is not a multiple of $p_0$, and all three eigenvalues identical when $k$ is a multiple of $p_0$. So it is either an LLL matrix or a scalar multiple of identity. 

Also note that roots of an LLL matrix must still be an LLL matrix. Therefore, powers and roots of scalar multiples of $A$ cannot be a pivot matrix. So the distance between the image of $A$ to the image of any pivot matrix is at least $2$.

Suppose $q$ is even. Suppose for contradiction that, in the reduce power graph of $\PGL_3(\F_q)$, the image of $A$ is connected to the image of $A'$ by an edge, and the image of $A'$ is connected to the image of a pivot matrix $B$ by an edge. Then $B$ must be a scalar multiple of a power of $A'$. Since $B$ cannot be a scalar multiple of a power of $A$, therefore $A'$ must be the root of a scalar multiple of $A$, and must be an LLL matrix.

Suppose $A'$ has eigenvalues $x',y',z'$, and $(A')^k=yA$ for some $y\in\F_q^*$ and some positive integer $k$. Then $(x')^k=y,(y')^k=yx,(z')^k=yx^2$. So $x'(y')^{-1},y'(z')^{-1},z'(x')^{-1}$ are all roots of $x$ or $x^2$, both are multiplicative generators of $\F_q^*$. Therefore $x'(y')^{-1},y'(z')^{-1},z'(x')^{-1}$ are all multiplicative generators of $\F_q^*$, and they all have multiplicative order $q-1$. So powers of $A'$ will either have all three eigenvalues distinct, or all three eigenvalues identical. So if $B$ is a scalar multiple of a power of $A'$, then it must have identical eigenvalues and become a multiple of identity. But $B$ is supposed to be a pivot matrix, contradiction.
\end{proof}

\begin{prop}
Let $q\neq 2$ be a power of a prime $p$. Suppose $q-1$ is not a prime power. If $q$ is odd, in the reduce power graph of $\PGL_3(\F_q)$, the distance between the image of $A$ and the image of $B$ is at least $8$. If $q$ is even, this distance is at least $10$. 
\end{prop}
\begin{proof}
Suppose for contradiction that there is a path from the image of $A$ to the image of $B$ in the reduce power graph of $\PGL_3(\F_q)$, with distance at most $7$ for $q$ odd or distance at most $9$ for $q$ even.

First note that $A,B$ do not have the same eigenspaces. So by Lemma~\ref{Lem:n3SameTypeEdge}, any path between them in the reduce power graph of $\PGL_3(\F_q)$ cannot be entirely made of LLL matrices. By Figure~\ref{Fig:n3q-1NotPowerJordanConn}, our path must contain a pivot matrix.

Let $A'$ be the first pivot matrix on this path starting from $A$, and let $B'$ be the last pivot matrix on this path. By Lemma~\ref{Lem:n3qNotPowerBadLLLExample}, the distance between $A$ and $A'$ is at least $2$ for $q$ odd, and at least $3$ for $q$ even. The same is true for $B$ and $B'$. Hence either way, the distance between $A'$ and $B'$ is at most $3$.

Now by construction, $A$ must be connected to $A'$ via a series of LLL matrices. By Lemma~\ref{Lem:n3LLLPivotEdge}, eigenspaces of $A'$ are spanned by eigenspaces of $A$. Let $\bs e_1,\bs e_2,\bs e_3$ be the three columns of the identity matrix, then the eigenspaces of $A'$ must be $\mathrm{span}(\bs e_1)$ and $\mathrm{span}(\bs e_2,\bs e_3)$, or $\mathrm{span}(\bs e_2)$ and $\mathrm{span}(\bs e_3,\bs e_1)$, or $\mathrm{span}(\bs e_3)$ and $\mathrm{span}(\bs e_1,\bs e_2)$. These are the only three possibilities.

Similarly, eigenspaces of $B'$ are spanned by eigenspaces of $B$. Let $\bs x_1,\bs x_2,\bs x_3$ be the three columns of the matrix $X$, then the one dimensional eigenspace of $B'$ must be $\mathrm{span}(\bs x_1)$ or $\mathrm{span}(\bs x_2)$ or $\mathrm{span}(\bs x_3)$. Neither of them can coinside with the one dimensional eigenspace of $A'$, and neither of them can be contained in the two dimensional eigenspace of $A'$. But this is impossible by Lemma~\ref{Lem:n3PivotPivotDistance3}. So we are done.
\end{proof}

\section{Connection restrictions to Jordan pivot matrices}

\begin{lem}
\label{Lem:n3JordanPivotLinePlane}
If $A\in\GL_3(\F_q)$ is a Jordan pivot matrix with eigenvalue $a$, then the column space of $A-aI$ is contained in the kernel of $A-aI$.
\end{lem}
\begin{proof}
This is because $(A-aI)^2$ is zero by the definition of a Jordan pivot matrix.
\end{proof}

\begin{lem}
\label{Lem:n3LLPJordanPivotEdge}
In the reduced power graph of $\PGL_3(\F_q)$, if the image of an LLP matrix $A$ and the image of a Jordan pivot matrix $B$ are connected by an edge. Suppose $A$ has simple eigenvalue $a_1\in\F_q^*$ and repeated eigenvalue $a_2\in\F_q^*$, and $B$ has an eigenvalue $b\in\F_q^*$. Then the column space of $B-bI$ is the eigenspace of $A$ for $a_2$, and the kernel of $B-bI$ is spanned by the two eigenspaces of $A$.
\end{lem}
\begin{proof}
If the image of an LLP matrix $A$ and the image of a Jordan pivot matrix $B$ are connected by an edge, then a power of $A$ must be $xB$ for some $x\in\F_q^*$. 

Since $A$ is an LLP matrix, $A=X\begin{bmatrix}a_1&&\\&a_2&1\\&&a_2\end{bmatrix}X^{-1}$ for some invertible $X$. Then $B=X\begin{bmatrix}x^{-1}a_1^k&&\\&x^{-1}a_2^k&x^{-1}ka_2^{k-1}\\&&x^{-1}a_2^k\end{bmatrix}X^{-1}$ for some positive integer $k$. Since $B$ is a Jordan pivot matrix with eigenvalue $b$, therefore all three diagonal entries here must be $b$.

Now the column space of $B-bI$ and the eigenspace of $A$ for $a_2$ are both the span of the second column of $X$. The two eigenspaces of $A$ are the span of the first column of $X$ and the span of the second column of $X$, while the kernel of $B-bI$ is spanned by the first two columns of $X$. So we are done.
\end{proof}

\begin{lem}
\label{Lem:n3NPJJordanPivotEdge}
In the reduced power graph of $\PGL_3(\F_q)$, if the image of an NPJ matrix $A$ and the image of a Jordan pivot matrix $B$ are connected by an edge. Suppose $A$ has an eigenvalue $a\in\F_q^*$ and $B$ has an eigenvalue $b\in\F_q^*$. Then the column space of $B-bI$ is the eigenspace of $A$ for $a$, and the kernel of $B-bI$ is the same as the kernel of of $(A-aI)^2$.
\end{lem}
\begin{proof}
If the image of an NPJ matrix $A$ and the image of a Jordan pivot matrix $B$ are connected by an edge, then $xB=A^k$ for some $x\in\F_q^*$ and some positive integer $k$, and $k$ must be an even number. 

Since $A$ is an NPJ matrix, $A=X\begin{bmatrix}a&1&\\&a&1\\&&a\end{bmatrix}X^{-1}$ for some invertible $X$. Then $B=X\begin{bmatrix}x^{-1}a^k&&\frac{k}{2}(k-1)a^{k-2}\\&x^{-1}a^k&\\&&x^{-1}a_2^k\end{bmatrix}X^{-1}$. Since $B$ is a Jordan pivot matrix with an eigenvalue $b$, therefore all three diagonal entries here must be $b$.

Now the column space of $B-bI$ and the eigenspace of $A$ for $a$ are both the span of the first column of $X$. The kernel of $B-bI$ and the kernel of $(A-aI)^2$ are both the span of the first two columns of $X$. So we are done.
\end{proof}

\begin{lem}
\label{Lem:n3JordanPivotJordanPivotPath}
In the reduced power graph of $\PGL_3(\F_q)$, suppose the image of a Jordan pivot matrix $A$ is connected to the image of a Jordan pivot matrix $B$ by an path that does not include any pivot matrix. Suppose $A$ has an eigenvalue $a\in\F_q^*$ and $B$ has an eigenvalue $b\in\F_q^*$. Then the column space of $A-aI$ is the same as the column space of $B-bI$, and the kernel of $A-aI$ is the same as the kernel of of $B-bI$.
\end{lem}
\begin{proof}
WLOG we can assume that there is no otherJordan pivot matrices on this path other than $A$ and $B$. By checking all diagrams from Figure~\ref{Fig:n3q-1NotPowerJordanConn} to Figure~\ref{Fig:n3q=2JordanConn}, all matrices in this path between $A$ and $B$ must have the same Jordan type. They must all be LLP matrices, or all be NJP matrices.

Suppose they are all LLP matrices. Then by Lemma~\ref{Lem:n3SameTypeEdge}, since all of them are connected by edges, they all have the same eigenspaces. By Lemma~\ref{Lem:n3LLPJordanPivotEdge}, the column space and the kernel of $A-aI$ and the column space and the kernel of $B-bI$ are determined by these same spaces. So we are done.

Suppose they are all NJP matrices. Then by Lemma~\ref{Lem:n3SameTypeEdge}, since all of them are connected by edges, they all have the same invariant subspaces. By Lemma~\ref{Lem:n3NPJJordanPivotEdge}, the column space and the kernel of $A-aI$ and the column space and the kernel of $B-bI$ are determined by these same spaces. So we are done.
\end{proof}

Now we investigate paths between a pivot matrix and a Jordan pivot matrix. We first make an important definition.

\begin{defn}
Let $A$ be a pivot matrix and $B$ be a Jordan pivot matrix with eigenvalue $b$. We say $A,B$ are compatible if the two dimensional eigenspace of $A$ contains the column space of $B-bI$, and the kernel of $B-bI$ contains the one dimensional eigenspace of $A$.
\end{defn}

\begin{lem}
\label{Lem:n3PivotJordanPivotPath}
Let $q\neq 2$ be a prime power such that $q-1$ is also a prime power. In the reduced power graph of $\PGL_3(\F_q)$, suppose the image of a pivot matrix $A$ is connected to the image of a Jordan pivot matrix $B$ by an path. Suppose either $A$ is the only pivot matrix on this path, or $B$ is the only Jordan pivot matrix on this path. Then $A,B$ are compatible to each other.
\end{lem}
\begin{proof}
Using Lemma~\ref{Lem:n3q9FermatPivotPivotPath} and Lemma~\ref{Lem:n3JordanPivotJordanPivotPath}, we can WLOG assume that $A$ is the only pivot matrix on this path, and $B$ is the only Jordan pivot matrix on this path. Then all other matrices on this path must be LLP matrices, and by Lemma~\ref{Lem:n3SameTypeEdge}, they must all have the same invariant subspaces. 

Suppose $C$ is an LLP matrix on this path with simple eigenvalue $c_1$ and repeated eigenvalue $c_2$. Let $L_1$ be the eigenspace of $C$ for $c_1$, $L_2$ be the eigenspace of $C$ for $c_2$, and $P$ be the generalized eigenspace of $C$ for $c_2$. By Lemma~\ref{Lem:n3LLPPivotEdge}, the eigenspaces for $A$ must be $L_1$ and $P$, while the column space of $B-bI$ is $L_2$, and the kernel of $B-bI$ is spanned by $L_1$ and $L_2$. It is easy to see now that $A,B$ are compatible to each other.
\end{proof}

\begin{lem}
\label{Lem:n3PivotPivotNotCompatiblePath}
Let $q\neq 2$ be a prime power such that $q-1$ is also a prime power. Let $A,B\in\GL_3(\F_q)$ be pivot matrices, such that the one dimensional eigenspace of $B$ is contained in the two dimensional eigenspace of $A$. Then in the reduced power graph of $\PGL_3(\F_q)$, their images have distance at least $8$ to each other.
\end{lem}
\begin{proof}
Suppose for contradiction that there is a path between them with distance at most $7$. Let the eigenspaces of $A$ be $L_A,P_A$ and the eigenspaces of $B$ be $L_B$ and $P_B$, where $L_A,L_B$ are one dimensional and $P_A$ and $P_B$ are two dimensional. Since $P_A$ cannot contain $L_A$, yet it contains $L_B$, therefore $A$ and $B$ do not have the same eigenspaces. So by Lemma~\ref{Lem:n3q9FermatPivotPivotPath}, there must be a Jordan pivot matrix on this path.

Let $A'$ be the first Jordan pivot matrix on this path starting from $A$, and let $B'$ be the last Jordan pivot matrix on this path. By Figure~\ref{Fig:n3q-1OddPowerJordanConn}, Figure~\ref{Fig:n3q-1EvenPowerJordanConn} and Figure~\ref{Fig:n3q=3JordanConn}, the distance between a pivot matrix and a Jordan pivot matrix is at least $2$. Since our path has total distance $7$, the distance between $A'$ and $B'$ is at most $3$. In particular, the path from $A'$ to $B'$ cannot contain a pivot matrix. If $A'$ has eigenvalue $a$ and $B'$ has eigenvalue $b$, then by Lemma~\ref{Lem:n3JordanPivotJordanPivotPath}, $A'-aI$ and $B'-bI$ will have the same column space and same kernel. Let the common column space be $L$ and the common kernel be $P$.

Since $A$ has a path to $A'$ on which $A'$ is the only Jordan pivot matrix, $A$ and $A'$ are compatible by Lemma~\ref{Lem:n3PivotJordanPivotPath}. Similarly, $B$ and $B'$ are compatible. So $L\subseteq P_A\cap P_B$, and therefore $L\neq L_A,L_B$. Now, since $L$ and $L_B$ are two different one dimensional subspaces inside $P_A$, $P_A$ must be the span of them. But similarly, $L_B$ must be inside $P$, and by Lemma~\ref{Lem:n3JordanPivotLinePlane}, $L$ must also be inside $P$. Hence $P$ is also the span of $L$ and $L_B$, and this means $P=P_A$. Then $P$ cannot contains $L_A$, contradiction. So we are done.
\end{proof}

\begin{lem}
\label{Lem:n3PivotJordanPivotNotCompatiblePath}
Let $q\neq 2$ be a prime power such that $q-1$ is also a prime power. Let $A\in\GL_3(\F_q)$ be a pivot matrix and let $B\in\GL_3(\F_q)$ be a Jordan pivot matrix with eigenvalue $b$, such that the two dimensional eigenspace of $A$ is the same as the kernel of $B-bI$. Then in the reduced power graph of $\PGL_3(\F_q)$, their images have distance at least $10$ to each other.
\end{lem}
\begin{proof}
Suppose there is a path from $A$ to $B$. Let $B'$ be the first pivot matrix on this path starting from $B$. By Lemma~\ref{Lem:n3PivotJordanPivotPath}, $B$ and $B'$ are compatible. So the one dimensional subspace of $B'$ is contained in the kernel of $B-bI$, which is the two dimensional eigenspace of $A$. By Lemma~\ref{Lem:n3PivotPivotNotCompatiblePath}, $A$ and $B'$ have distance at least $8$, while the path from $B'$ to $B$ has distance at least $2$ by Figure~\ref{Fig:n3q-1OddPowerJordanConn}, Figure~\ref{Fig:n3q-1EvenPowerJordanConn} and Figure~\ref{Fig:n3q=3JordanConn}. Therefore this path from $A$ to $B$ must have a distance at least $10$.
\end{proof}

\section{Lower bounds when $q-1$ is a prime power}

Let $q\neq 2$ be a power of a prime $p$, such that $q-1$ is also the power of a prime.

If $q\neq 3$ and $q-1$ is a power of $2$, $q+1$ cannot be a power of $2$. Pick any odd prime factor $p_0$ of $q+1$. If $q\neq 2$ is a power of $2$ and $q-1$ is a prime, then $q+1$ must be coprime to $q-1$. Pick any prime factor $p_0$ of $q+1$. Finally, if $q=3$, then set $p_0=q^2-1=8$. Either way, $p_0$ is a factor of $q^2-1$ but not a factor of $q-1$.

Let $C$ be a companion matrix to any irreducible polynomial over $\F_q$ of degree $2$, then $\F_q[C]$ is a field with $q^2$ elements. So we can find $C'\in\F_q[C]$ with multiplicative order $p_0$. Since $p_0$ cannot divide $q-1$, $C'$ must have irreducible characteristic polynomial. Set $A=\begin{bmatrix}C'&\\&1\end{bmatrix}$, then $A$ is an LP matrix.

\begin{lem}
\label{Lem:n3qPowerBadLPExample}
Let $q\neq 2,3$ be a power of a prime $p$. Suppose $q-1$ is a prime power. Then in the reduce power graph of $\PGL_3(\F_q)$, the image of $A$ has distance at least $2$ to the image of any pivot matrix.
\end{lem}
\begin{proof}
Since $q\neq 2,3$, $p_0$ is a prime number. So by Corollary~\ref{Cor:CompanionPower}, powers of $A$ are either still LP matrices or scalar multiples of identity. We can also see that roots of scalar multiples of $A$ cannot be a pivot matrix. So $A$ has distance at least $2$ to the image of any pivot matrix.
\end{proof}

\begin{prop}
Let $q\neq 3$ be a Fermat prime or $q=9$. Then the reduce power graph of $\PGL_3(\F_q)$ has diameter at least $12$. 
\end{prop}
\begin{proof}
Set $B=\begin{bmatrix}1&&1\\&1&\\&&1\end{bmatrix}$, so it is a Jordan pivot matrix. Suppose there is a path from $A$ to $B$. By Figure~\ref{Fig:n3q-1OddPowerJordanConn}, this path must contain a pivot matrix. 

Let $A'$ be the first pivot matrix on this path starting from $A$. Let $\bs e_1,\bs e_2,\bs e_3$ be the three columns of the identity matrix. By Lemma~\ref{Lem:n3LPPivotEdge}, the eigenspace of $A'$ must be the span of $\bs e_1$ and $\bs e_2$, and the span of $\bs e_3$. In particular, the two dimensional eigenspace of $A'$ is the same as the kernel of $B-I$. By Lemma~\ref{Lem:n3PivotJordanPivotNotCompatiblePath}, the images of $A'$ and $B$ has distance at least $10$.

By Lemma~\ref{Lem:n3qPowerBadLPExample}, the images of $A$ and $A'$ has distance at least $2$. So this path from $A$ to $B$ have distance at least $12$.
\end{proof}

\begin{prop}
Let $q\neq 2$ be a power of $2$, and $q-1$ is prime. Then the reduce power graph of $\PGL_3(\F_q)$ has diameter at least $13$. 
\end{prop}
\begin{proof}
Set $B=\begin{bmatrix}1&1&\\&1&1\\&&1\end{bmatrix}$, so it is an NPJ matrix. Suppose there is a path from $A$ to $B$. By Figure~\ref{Fig:n3q-1EvenPowerJordanConn}, this path must contain a pivot matrix and a Jordan pivot matrix. 

Let $A'$ be the first pivot matrix on this path starting from $A$. Let $\bs e_1,\bs e_2,\bs e_3$ be the three columns of the identity matrix. By Lemma~\ref{Lem:n3LPPivotEdge}, the eigenspace of $A'$ must be the span of $\bs e_1$ and $\bs e_2$, and the span of $\bs e_3$. 

Let $B'$ be the last Jordan pivot matrix on this path starting from $A$, and suppose it has an eigenvalue $b\in\F_q^*$. Let $\bs e_1,\bs e_2,\bs e_3$ be the three columns of the identity matrix. By Lemma~\ref{Lem:n3NPJJordanPivotEdge}, the kernel of $B'-bI$ must be the span of $\bs e_1$ and $\bs e_2$, and it is identical to the two dimensional eigenspace of $A'$. By Lemma~\ref{Lem:n3PivotJordanPivotNotCompatiblePath}, the images of $A'$ and $B'$ has distance at least $10$.

By Lemma~\ref{Lem:n3qPowerBadLPExample}, the images of $A$ and $A'$ has distance at least $2$. And since $B$ and $B'$ are not in the same Jordan type, their images have distance at least $1$. So this path from $A$ to $B$ have distance at least $13$.
\end{proof}

\begin{prop}
The reduce power graph of $\PGL_3(\F_3)$ has diameter at least $11$. 
\end{prop}
\begin{proof}
Set $B=\begin{bmatrix}1&&1\\&1&\\&&1\end{bmatrix}$, so it is a Jordan pivot matrix. Suppose there is a path from $A$ to $B$. By Figure~\ref{Fig:n3q=3JordanConn}, this path must contain a pivot matrix. 

Let $A'$ be the first pivot matrix on this path starting from $A$. Let $\bs e_1,\bs e_2,\bs e_3$ be the three columns of the identity matrix. By Lemma~\ref{Lem:n3LPPivotEdge}, the eigenspace of $A'$ must be the span of $\bs e_1$ and $\bs e_2$, and the span of $\bs e_3$. In particular, the two dimensional eigenspace of $A'$ is the same as the kernel of $B-I$. By Lemma~\ref{Lem:n3PivotJordanPivotNotCompatiblePath}, the images of $A'$ and $B$ has distance at least $10$.

Since $A$ and $A'$ are not in the same Jordan type, their images have distance at least $1$. So this path from $A$ to $B$ have distance at least $11$.
\end{proof}

\bibliographystyle{alpha}
\bibliography{references}

\begin{thebibliography}{AA15}

\bibitem[AA15]{AA15}
Narges Akbari and Ali-Reza Ashrafi.
\newblock Note on the power graph of finite simple groups.
\newblock {\em Quasigroups and Related Systems}, 23(2):165--173, 2015.

\bibitem[Yan]{previous}
Yilong Yang.
\newblock Connectivity of the reduced power graphs of finite simple groups.
\newblock {\em preprint https://arxiv.org/abs/2207.04404}.

\end{thebibliography}

\end{document}